\documentclass[11pt]{article}

\usepackage[a4paper,margin=1in]{geometry}
\usepackage{amsmath}
\usepackage{amssymb}
\usepackage{amsthm}
\usepackage{booktabs}
\usepackage{array}
\usepackage[numbers,sort&compress]{natbib}
\usepackage{hyperref}

\newtheorem{theorem}{Theorem}
\newtheorem{proposition}{Proposition}
\newtheorem{corollary}[proposition]{Corollary}
\newtheorem{lemma}{Lemma}

\title{An Exhaustive Census of Main-Diagonal Symmetric Costas Arrays of Orders 37--42}
\author{Bogdan Dumitru\thanks{Faculty of Mathematics and Computer Science, University of Bucharest, Romania},
Cristian Budala\footnotemark[1]}
\date{}

\begin{document}

\maketitle

\begin{abstract}
We determine by exhaustive enumeration all main-diagonal symmetric Costas
arrays of orders 37 through 42.  There are 4 arrays of order 37, none of
order 38, 16 of order 39, 2 of order 40, 12 of order 41, and 4 of order
42.  All belong to known finite-field constructions.  The arrays at
orders 37--40 are Lempel
arrays over $\mathbb{F}_{41}$ or their corner deletions and augmentations;
the twelve arrays at order 41 are the Lempel arrays over
$\mathbb{F}_{43}$; and the order-42 census consists of two corner
augmentations and a reverse-complement pair in the Rickard--Golomb family.
Combined with the published census through order 36, these results
show that no sporadic main-diagonal symmetric Costas array occurs at
orders 24--42.  We also develop a collision model for random involutions.
The first two terms of its collision exponent are derived analytically,
$n^2/18+n^{3/2}/360$, with an $O(n)$ remainder.  The lower-order model
and clumping correction are empirical.  The six new censuses are
consistent with the model's prediction of a rapidly declining expected
sporadic population.
\end{abstract}

\section{Introduction}
\label{sec:intro}

An order-$n$ Costas array is an $n\times n$ permutation matrix in which the
$\binom{n}{2}$ displacement vectors between pairs of dots are pairwise
distinct \cite{Costas1984,GolombTaylor1984,Drakakis2006Review}.  Written as
a permutation $p=(p_0,\ldots,p_{n-1})$, the condition states that for every
$d\geq 1$ the differences $p_{i+d}-p_i$ are distinct.  A Costas array is
\emph{main-diagonal symmetric} if its matrix is symmetric; equivalently,
the permutation is an involution, $p\circ p=\mathrm{id}$.

The exhaustive census of several Costas-array symmetry types by Taylor,
Rickard and Drakakis \cite{RickardDrakakis2009Symmetry} reached order 36.
We call a Costas array \emph{sporadic} if it does not belong to any
construction family recognized in that census or in the classification below.
Sporadic arrays occur at orders 11 through 23, while
every array at orders 24 through 36 comes from a known finite-field
construction.  The census closed by asking more broadly whether a symmetric
Costas array of unknown origin exists above order 23.

We address the main-diagonal case by exhaustive enumeration of six orders.
The main-diagonal symmetric Costas arrays of orders 37, 38, 39 and 40 number
$4, 0, 16, 2$, and none of them is new: all are Lempel arrays over
$\mathbb{F}_{41}$ and their corner edits.  Order 38, which no field
construction reaches, is empty.  The main-diagonal symmetric Costas
arrays of order 41 are exactly the twelve Lempel arrays over
$\mathbb{F}_{43}$.  The census of order 42 contains four arrays, all
members of the family $c\,(g^{i}+g^{j})=1$ over $\mathbb{F}_{43}$: the
two corner augmentations of the Lempel family and a reverse-complement
pair in the Rickard--Golomb family
\cite{Rickard2004Periodicity,Drakakis2011OpenProblems}.  So no sporadic
main-diagonal symmetric Costas array exists at any order from 24 through 42.

The search uses three consequences of symmetry, proved in
Section~\ref{sec:lemmas}.  No two involution orbits share the sum $a+b$,
no two transpositions share the gap $b-a$, and the fixed points form a
\emph{Golomb ruler}, a set of positions with pairwise distinct
differences.

The censuses also test a heuristic model.  For a uniform random
involution, we compute the expected number of displacement collisions
after transpose-paired classes are counted once.  Section~\ref{sec:model}
proves
\[
\lambda_{\mathrm{sym}}(n)
= \frac{(n-1)(n-2)}{18}+\frac{n^{3/2}}{360}+O(n).
\]
The $O(n)$ remainder is calibrated over the measured range rather than
claimed as an asymptotic expansion.  We compare the resulting count
with exact enumeration of all involutions up to $n=18$ and with Monte
Carlo estimates up to $n=60$.  The expected collision count grows quadratically,
whereas the number of involutions $I(n)$ satisfies
$\ln I(n)\sim\tfrac12 n\ln n$.  Under the model, the expected sporadic
population therefore tends to zero.

Depending on the calibration, the model places the crossing at which the
expected sporadic population falls below one between orders 27 and 30.
The census record is consistent with this estimate: no sporadic main-diagonal
symmetric array occurs after order 23, and order 38 is empty.  For general Costas
arrays, exponential decay of the density is a theorem
\cite{WarnkeCorrellSwanson2023}; the model here is a heuristic for the
symmetric class.

Separate CPU and CUDA implementations completely enumerated orders
37--40 and returned identical array sets.  This agreement provides an
implementation-level cross-check of the census.  All solvers were
validated against the published enumeration record.
The appendices give the validation protocols, execution records, and
negative results; the companion repository contains the source code,
censuses, and compact verification records.

\section{Involutions as linear forests}
\label{sec:forest}

Fix an order-$n$ symmetric Costas array with permutation $p$.  Since the
matrix is symmetric, $p$ is an involution; it decomposes into $f$ fixed
points $p_a=a$ and $t$ transpositions $\{a,b\}$, $p_a=b$, $p_b=a$, $a<b$,
with $n=f+2t$.  In particular $f$ has the parity of $n$.

Map the array to a graph $G$ on vertex set $\{0,1,\ldots,n\}$: one edge
per orbit, the orbit $\{a,b\}$ ($a\leq b$) becoming the edge $(a,\,b+1)$.
A fixed point at $a$ is the unit edge $(a,a+1)$; a transposition of gap
$g=b-a$ is an edge of length $g+1$.  The map is a bijection between
involutions on $n$ points and the resulting edge sets.

\begin{proposition}
\label{prop:forest}
$G$ is a disjoint union of simple paths, single vertices included.  Every vertex has degree at most
two, vertices $0$ and $n$ have degree at most one, and $G$ contains no
cycle.
\end{proposition}

\begin{proof}
A vertex $v$ is the left endpoint of at most one edge (the orbit whose
minimum is $v$) and the right endpoint of at most one edge (the orbit
whose maximum is $v-1$), so degrees are at most two, and the two boundary
vertices lack one side each.  Suppose a cycle existed and let $v$ be its
smallest vertex.  Both cycle edges at $v$ go rightward, so $v$ is the left
endpoint of two edges, i.e.\ the minimum of two orbits --- impossible.
\end{proof}

The forest has $(n+f)/2$ edges and therefore determines $f$.  The reverse complement
$RC(p)_i = n-1-p_{n-1-i}$ is induced by $180^\circ$ rotation and gives
the only nontrivial action of the square's symmetry group on this class.
On $G$, it acts as the mirror reflection $v\mapsto n-v$.
Canonicalizing a search under $RC$ is
therefore the same as canonicalizing a forest under reflection; we use
this form in Section~\ref{sec:method}.

\section{Three symmetry constraints}
\label{sec:lemmas}

The dot set $S=\{(i,p_i)\}$ of a symmetric array is closed under
transposition $(i,j)\mapsto(j,i)$, and the transpose of a displacement
vector $(u,v)$ is $(v,u)$.  So displacement vectors occur in mirror pairs:
if $v$ occurs between two dots, $v^{T}$ occurs between the transposed
dots.  The fixed classes of $w\mapsto w^{T}$ give the symmetry-specific
conditions used below.

\begin{lemma}[anti-diagonal vectors are internal]
\label{lem:antidiag}
In a symmetric Costas array, every displacement vector of the form
$(u,-u)$ connects the two dots of a single transposition.  Consequently:
(a) distinct orbits have distinct center sums $a+b$; (b) distinct
transpositions have distinct gaps $b-a$.
\end{lemma}

\begin{proof}
Let $u>0$ and suppose that $w=(u,-u)$ occurs from dot $P$ to dot $Q$.
Transposition gives $Q^{T}-P^{T}=w^{T}=-w$, or equivalently
$P^{T}-Q^{T}=w$.  Thus $w$ also occurs, in increasing-row order, from
$Q^{T}$ to $P^{T}$.  Uniqueness of displacement vectors forces
$(Q^{T},P^{T})=(P,Q)$, and hence $Q=P^{T}$: the two dots belong to one
transposition.

For (a), represent two distinct orbits by $\{a,b\}$ and $\{c,d\}$, with
$a\leq b$, $c\leq d$, and $a+b=c+d$.  After exchanging the orbits if
necessary, $a<c$, and the dots $(a,b)$ and $(c,d)$ determine the vector
\[
(c-a,d-b)=(c-a,-(c-a)),
\]
which joins distinct orbits, contradicting the first part.  For (b), two
transpositions of equal gap $g$ each carry the internal vector $(g,-g)$,
again contradicting uniqueness.
\end{proof}

\begin{lemma}[diagonal vectors join fixed points]
\label{lem:diag}
Every displacement vector of the form $(u,u)$ connects two fixed points.
Consequently the fixed-point positions form a Golomb ruler: their pairwise
differences are distinct.
\end{lemma}

\begin{proof}
Let $u>0$ and suppose that $w=(u,u)$ occurs from dot $P$ to dot $Q$.
Since $w^{T}=w$, the same vector occurs from $P^{T}$ to $Q^{T}$, also in
increasing-row order.  Uniqueness of displacement vectors forces
$(P^{T},Q^{T})=(P,Q)$.  Hence $P=P^{T}$ and $Q=Q^{T}$, so both dots lie on
the main diagonal.  Two pairs of fixed points with equal spacing would
repeat the corresponding vector $(u,u)$; therefore their positions form
a Golomb ruler.  The Golomb-ruler property for diagonal dots of arbitrary
Costas arrays is also known
\cite{DrakakisGowOCarroll2006,Drakakis2009GolombRulers,
Drakakis2011OpenProblems}; symmetry gives
the stronger assertion proved here that a diagonal vector cannot join
off-diagonal dots.
\end{proof}

For vectors with $u\neq\pm v$, transposition pairs two distinct vector
classes.  Proposition~\ref{prop:complete} shows that the only
self-conjugate classes are represented by $(u,u)$ and $(u,-u)$.

\begin{proposition}[transpose-orbit classification]
\label{prop:complete}
For a main-diagonal symmetric permutation, transpose partitions the
displacement-vector classes, with a vector and its negative identified,
into two-element orbits except for the classes represented by $(u,u)$ and
$(u,-u)$.  The Costas conditions for the two classes in a two-element
orbit are equivalent.  The conditions for the exceptional classes are
exactly those of Lemmas~\ref{lem:antidiag} and~\ref{lem:diag}: distinct
orbit center sums, distinct transposition gaps, and the fixed-point Golomb
ruler.
\end{proposition}

\begin{proof}
Transposition sends an occurrence of $w=(u,v)$ to an occurrence of
$w^{T}=(v,u)$.  For $v=\pm u$ the vector class is its own mirror,
occurrences pair within the class, and the two lemmas describe the forced
structure --- the three constraints.  For $v\neq\pm u$, the classes
represented by $w$ and $w^{T}$ are distinct.  Transposition gives a
bijection between their occurrence sets, so their Costas uniqueness
conditions are equivalent.  The only self-conjugate classes are therefore
those covered by the two lemmas.
\end{proof}

Related prior work includes the mirror-pair analysis of Jedwab and
Wodlinger \cite{JedwabWodlinger2014}.  They study mirror pairs $(w,h)$,
$(w,-h)$ in arbitrary and $G$-symmetric Costas arrays through vertical
reflection, whereas the lemmas above use transpose closure for
main-diagonal symmetry.  The fixed-point Golomb-ruler consequence was
observed earlier by Drakakis, Gow and O'Carroll
\cite{DrakakisGowOCarroll2006}.  See also the difference-based structural
constraints of \cite{CorrellSwanson2023}.  The ruler condition implies that the
difference $1$ occurs at most once among the fixed points.  Thus a
symmetric Costas array never contains three consecutive fixed points and
contains at most one adjacent pair.  At orders 41 and 42 we use the
stronger ruler bound below.

\begin{corollary}
\label{prop:ruler}
The number of fixed points of an order-$n$ symmetric Costas array is at
most the largest order of a Golomb ruler of span $n-1$; in particular
$f=O(\sqrt n)$.  For $n=41$ and $n=42$ the bound is $f\leq 8$, leaving the
classes $f\in\{1,3,5,7\}$ at order 41 and $f\in\{0,2,4,6,8\}$ at order 42.
\end{corollary}

\begin{proof}
Let $0\leq x_1<\cdots<x_f\leq n-1$ be the fixed-point positions.
Lemma~\ref{lem:diag} shows that the $\binom{f}{2}$ positive differences
$x_j-x_i$, $i<j$, are distinct.  Each belongs to
$\{1,\ldots,n-1\}$, so
\[
\binom{f}{2}\leq n-1.
\]
It follows that $f=O(\sqrt n)$.

For the two orders considered here, the elementary inequality still
allows nine fixed points.  The optimal nine-mark Golomb ruler has span
44 \cite{Dimitromanolakis2002GolombRulers}.  At order 41
the fixed-point positions lie in $\{0,\ldots,40\}$, of span 40, and at
order 42 they lie in $\{0,\ldots,41\}$, of span 41.  Thus neither order
can have nine fixed points, and $f\leq8$.  Finally, $n=f+2t$ implies that
$f$ has the parity of $n$, giving the stated classes.
\end{proof}

\section{A first moment for symmetric Costas arrays}
\label{sec:model}

Silverman et al.\ fitted a probabilistic estimate for the number of Costas
arrays to enumeration data in 1988 \cite{SilvermanVickersMooney1988}.
Warnke et al.\ later proved that the density $C(n)/n!$ decays exponentially
\cite{WarnkeCorrellSwanson2023}.  This section derives the expected collision
count for a uniform random involution and uses a zero-collision approximation
to model the sporadic count.  The quadratic and $n^{3/2}$ terms of the
collision exponent are derived below.  The probability approximation,
linear-scale refinement and clumping correction are heuristic or fitted.

\subsection{Mirror pairing and the exponent}

Let $p$ be a uniform random involution.  For each class $(k,\delta)$, with
positive row stride $k$ and signed column difference $\delta$, let its
occupancy be the number of index pairs that produce that displacement.
Let $W$ be the sum, over all classes, of the number of pairs of
occurrences in the class.  Displacements are identified with their
negatives when their orientation is normalized to positive row stride.
Transpose closure pairs the classes: a collision in $(k,\delta)$ occurs
if and only if the corresponding collision occurs in the normalized
class represented by $(\delta,k)$.  The two are the same event exactly
for the self-conjugate classes $\delta=\pm k$ ---
which by Lemmas~\ref{lem:antidiag} and~\ref{lem:diag} are the center, gap
and ruler constraints.  Write $W_{\mathrm{full}}$ for $W$ and
$W_{\mathrm{self}}$ for its restriction to the self-conjugate classes.
Since the two classes in each non-self-conjugate mirror pair have identical
collision counts, we count the pair once and retain each self-conjugate class
in full:
\[
\lambda_{\mathrm{sym}}
= \mathbb{E}\!\left[\tfrac{1}{2}\bigl(W_{\mathrm{full}}-W_{\mathrm{self}}\bigr)
+ W_{\mathrm{self}}\right].
\]

\subsection{Bulk term}

\begin{lemma}
\label{lem:involution-ratio}
Let $I_n$ be the number of involutions on $n$ labeled points.  For every
fixed nonnegative integer $r$,
\[
\frac{I_{n-r}}{I_n}
=n^{-r/2}\left(1-\frac{r}{2\sqrt n}+O(n^{-1})\right).
\]
\end{lemma}

\begin{proof}
Put $q_n=I_{n-1}/I_n$.  The recurrence
$I_n=I_{n-1}+(n-1)I_{n-2}$ gives
\[
q_n=\frac{1}{1+(n-1)q_{n-1}}.
\]
Applying this recurrence twice gives
\[
q_n=G_n(q_{n-2}),
\qquad
G_n(x)=\frac{1+(n-2)x}{n+(n-2)x}.
\]
Set $a_n=n^{-1/2}-1/(2n)$.  Direct expansion gives
\[
G_n(a_{n-2})=a_n+\frac{3}{4n^2}+O(n^{-5/2}),
\]
while the exact identity
\[
G_n(x)-G_n(y)
=\frac{(n-1)(n-2)(x-y)}
{[n+(n-2)x][n+(n-2)y]}
\]
shows that, uniformly for $x=a_{n-2}+O(n^{-3/2})$,
\[
G_n(x)-G_n(a_{n-2})
=\left(1-\frac{2}{\sqrt n}+O(n^{-1})\right)(x-a_{n-2}).
\]
Choose a sufficiently large $N$, and then choose $C$ so that the bound
$|q_m-a_m|\leq Cm^{-3/2}$ holds for $m=N,N+1$.  If it holds at $n-2$,
the two estimates above give
\[
|q_n-a_n|
\leq
\left(1-\frac{2}{\sqrt n}+O(n^{-1})\right)C(n-2)^{-3/2}
+\frac{3}{4n^2}+O(n^{-5/2})
\leq Cn^{-3/2}
\]
for sufficiently large $N$ and $C$.  Induction on the even and odd
indices therefore gives
\[
q_n=n^{-1/2}-\frac{1}{2n}+O(n^{-3/2}).
\]
Since $r$ is fixed,
$I_{n-r}/I_n=q_nq_{n-1}\cdots q_{n-r+1}$, and multiplying these $r$
expansions proves the claim.
\end{proof}

\begin{proposition}
\label{prop:bulk}
\[
\mathbb{E}[W_{\mathrm{full}}]
=\frac{n^2}{9}+\frac{n^{3/2}}{180}+O(n).
\]
\end{proposition}

\begin{proof}
For a fixed stride $k$, a collision comparison has the form
\[
p(i+k)-p(i)=p(j+k)-p(j),\qquad i<j.
\tag{1}
\]
There are
$\sum_{k=1}^{n-1}\binom{n-k}{2}=\binom n3$ such comparisons.  The cases
$j=i+k$, in which only three row positions occur, number $O(n^2)$ and
will be included in the remainder.

\emph{Eight vertices.}
First suppose that the four row positions and their four images are
disjoint.  For each row comparison, the unrestricted number of image
quadruples satisfying (1) is
\[
\sum_{d\ne0}(n-|d|)^2
=\frac{n(n-1)(2n-1)}{3}
=\frac23n^3+O(n^2).
\]
Requiring the images to be distinct and disjoint from the rows removes
only $O(n^2)$ quadruples.  Hence the number $A_8(n)$ of compatible
eight-vertex specifications is
\[
A_8(n)=\frac19n^6+O(n^5).
\]
Every such partial involution extends in $I_{n-8}$ ways, so
Lemma~\ref{lem:involution-ratio} gives
\[
A_8(n)\frac{I_{n-8}}{I_n}
=\frac19n^2-\frac49n^{3/2}+O(n).
\tag{2}
\]

\emph{Seven vertices.}
A compatible specification on exactly seven vertices has one fixed row
position and pairs the other three rows with external vertices.  If the
fixed position is $x$, first ignore assignments that reuse a row or an
image.  The number of assignments of the other three images is
\[
F_n(x)=\sum_{d=-x}^{n-1-x}(n-|d|)
=\frac{n^2}{2}+nx+\frac n2-x^2-x.
\tag{3}
\]
We now sum this polynomial over
$1\leq k\leq n-1$ and $0\leq i<j\leq n-k-1$.  The standard sums of
powers give
\begin{align*}
\sum_{k,i,j}F_n(i)
=\sum_{k,i,j}F_n(j+k)
&=\frac{n(n-1)(n-2)(13n^2-n+6)}{120},\\
\sum_{k,i,j}F_n(i+k)
=\sum_{k,i,j}F_n(j)
&=\frac{n(n-1)(n-2)(7n^2+n+4)}{60}.
\end{align*}
Thus the four possible fixed positions give
\[
\frac{n(n-1)(n-2)(27n^2+n+14)}{60}
=\frac9{20}n^5+O(n^4).
\]
For each row comparison, forbidding a reused row or image removes only
$O(n)$ assignments; the comparisons with an overlapping row contribute
only $O(n^4)$ to the displayed unrestricted sum.  Therefore the number
of compatible seven-vertex specifications is
\[
A_7(n)=\frac9{20}n^5+O(n^4).
\]
Lemma~\ref{lem:involution-ratio} now gives
\[
A_7(n)\frac{I_{n-7}}{I_n}
=\frac9{20}n^{3/2}+O(n).
\tag{4}
\]

\emph{Remaining patterns.}
For four distinct row positions, let $f$ be the number that are fixed,
$t$ the number of internal transpositions among them, and $e$ the number
paired with external vertices.  Then
\[
f+2t+e=4,
\qquad\text{and the partial involution uses }4+e\text{ vertices}.
\]
The cases $e=4$ and $e=3$ are (2) and (4).  If $0<e\leq2$, equation (1)
is a nonzero affine equation in the $e$ external labels.  There are
$O(n^{e-1})$ label assignments for each row comparison, and the total
contribution is
\[
O\!\left(n^3 n^{e-1} n^{-(4+e)/2}\right)=O(n^{e/2})=O(n).
\]
For $e=0$, there are $O(n^3)$ specifications, each with probability
$O(n^{-2})$, again contributing $O(n)$.

When $j=i+k$, only three row positions occur and there are $O(n^2)$ row
comparisons.  If $e>0$ of the three rows are paired externally, the
arithmetic-progression equation among their images leaves
$O(n^{e-1})$ external assignments.  Since the partial involution uses
$3+e$ vertices and $e\leq3$, the contribution is
\[
O\!\left(n^2n^{e-1}n^{-(3+e)/2}\right)=O(n^{(e-1)/2})=O(n).
\]
The case $e=0$ contributes
$O(n^2n^{-3/2})=O(n^{1/2})$.

Combining (2) and (4), with all other patterns absorbed into $O(n)$,
gives
\[
\mathbb{E}[W_{\mathrm{full}}]
=\frac19n^2+
\left(-\frac49+\frac9{20}\right)n^{3/2}+O(n)
=\frac19n^2+\frac1{180}n^{3/2}+O(n).
\]
\end{proof}

\subsection{Self-conjugate term and finite-range refinement}

\begin{proposition}
\label{prop:self-bound}
\[
\mathbb{E}[W_{\mathrm{self}}]=O(n).
\]
\end{proposition}

\begin{proof}
For four distinct row positions, retain the notation $e$ from the proof
of Proposition~\ref{prop:bulk}.  In a self-conjugate collision both
differences in (1) equal $\sigma k$, where $\sigma\in\{-1,1\}$.  For each
fixed row comparison and orbit pattern, the following table gives an
upper bound on the number of external-label assignments.  The contribution
row also includes the $O(n^3)$ choices of the row comparison and the factor
$O(n^{-(4+e)/2})$ from
Lemma~\ref{lem:involution-ratio}.
\[
\begin{array}{c|ccccc}
e&4&3&2&1&0\\ \hline
\text{external assignments}&O(n^2)&O(n)&O(n)&O(1)&O(1)\\
\text{contribution}&O(n)&O(n^{1/2})&O(n)&O(n^{1/2})&O(n)
\end{array}
\]
Indeed, a pair with two external images has $O(n)$ choices once their
difference is fixed; a pair with one external image determines it; and
a pair with no external image gives either a constraint on the row
positions or no compatible specification.

For the $O(n^2)$ comparisons with three row positions, three external
images form an arithmetic progression with prescribed difference and
therefore have only $O(n)$ choices.  With two, one, or no external images
there are $O(1)$ choices after the row positions are fixed.  Multiplying
by $O(n^{-(3+e)/2})$ shows that these four cases contribute respectively
$O(1)$, $O(n^{-1/2})$, $O(1)$, and $O(n^{1/2})$.  Summing the finitely
many orbit patterns and the two signs proves the proposition.
\end{proof}

\begin{theorem}[first-moment expansion]
\label{thm:first-moment}
\[
\lambda_{\mathrm{sym}}(n)
=\frac{n^2}{18}+\frac{n^{3/2}}{360}+O(n).
\]
\end{theorem}

\begin{proof}
The mirror-pair identity gives
\[
\lambda_{\mathrm{sym}}
=\tfrac12\mathbb{E}[W_{\mathrm{full}}]
+\tfrac12\mathbb{E}[W_{\mathrm{self}}].
\]
Apply Propositions~\ref{prop:bulk} and~\ref{prop:self-bound}.
\end{proof}

The theorem determines the first two terms but not the linear
coefficient.  A heuristic occupancy calculation for the self-conjugate
classes gives a linear contribution of $41n/30$; a separate fit gives
$1.362968n$.  Appendix~\ref{app:model-computation} gives the calculation,
data, and fit specification.  Proposition~\ref{prop:bulk} leaves an
additional $O(n)$ contribution from the non-self-conjugate part, so
$41/30$ is not imposed as the linear coefficient of the full exponent.
For numerical work over the measured range, we fit
\[
\lambda_{\mathrm{model}}(n)
=\frac{(n-1)(n-2)}{18}+\frac{n^{3/2}}{360}
+b_1n+b_2\sqrt n+b_3
\tag{5}
\]
to the same exact and Monte Carlo orders.  This gives
\[
b_1=1.300606,\qquad b_2=-2.467695,\qquad b_3=1.196849.
\]
The maximum absolute residual at the fitted orders is $0.044$.  Equation
(5) is a finite-range calibration of the theorem's $O(n)$ remainder, not
an asserted asymptotic expansion.
To evaluate this calibration, we computed $\lambda_{\mathrm{sym}}$ exactly
for orders up to 18 by enumerating all involutions.  At order 18, their
number is $I(18)=997{,}313{,}824$.  For orders 19 through 60, we used Monte
Carlo with $K$ up to $2\times 10^{6}$ and standard errors below $0.08$.
Table~\ref{tab:lambda} compares the assembled model with the measurements.

\begin{table}[ht]
\centering
\begin{tabular}{rrrr}
\toprule
$n$ & $\lambda_{\mathrm{sym}}$ (exact/MC) & model & residual \\
\midrule
12 & 14.477 & 14.482 & $-0.005$ \\
14 & 18.984 & 18.984 & $0.000$ \\
16 & 23.984 & 23.980 & $+0.003$ \\
18 & 29.467 & 29.461 & $+0.005$ \\
24 & 48.762 & 48.760 & $+0.002$ \\
32 & 81.051 & 81.026 & $+0.024$ \\
40 & 120.672 & 120.650 & $+0.022$ \\
48 & 167.567 & 167.564 & $+0.003$ \\
\bottomrule
\end{tabular}
\caption{The first-moment exponent against the finite-range model (5).
The derived terms are $(n-1)(n-2)/18+n^{3/2}/360$; the fitted
coefficients are $b_1=1.300606$, $b_2=-2.467695$, and $b_3=1.196849$.
Exact values for $n\leq18$; Monte Carlo beyond
($K=2\times10^{6}$ at $n\geq40$).}
\label{tab:lambda}
\end{table}

For the non-self-conjugate part, subtracting the derived $n^2$ and
$n^{3/2}$ terms leaves residuals of $-0.1553n$, $-0.1550n$, and
$-0.1546n$ at orders 40, 48, and 60.  This is consistent with the $O(n)$
remainder in Theorem~\ref{thm:first-moment}.

\subsection{Clumping and extinction boundary}

The Poisson heuristic predicts a sporadic population
$I(n)\,e^{-\lambda_{\mathrm{sym}}(n)}$.  Let
$C_{\mathrm{sym}}(n)$ denote the total number of main-diagonal symmetric
Costas arrays of order $n$.  The published sporadic-only series contains
only 1--15 equivalence classes per order at orders 11--23, which is too
sparse for a stable three-parameter fit
\cite{RickardDrakakis2009Symmetry}.  We therefore calibrate the clumping
correction on the larger total census counts at orders 10--23.  Write
\[
r_{\mathrm{fit}}(n)
=0.035827520\,n^2-0.124784466\,n+2.612907856.
\]
Using the higher-sample estimates, the three-term fit gives
\[
\ln\frac{C_{\mathrm{sym}}(n)}{I(n)\,e^{-\lambda_{\mathrm{sym}}(n)}}
\;\approx\; r_{\mathrm{fit}}(n),
\qquad n = 10\text{--}23 .
\]
Its quadratic coefficient, reported as $0.036$, is an empirical finite-range
calibration, not a constant derived from a collision calculation.
Appendix~\ref{app:model-computation} gives the sensitivity fits.
Collisions clump: configurations that produce one collision tend to
produce several, so the number of \emph{distinct} bad configurations is
smaller than the collision count.  Using the published general Costas census
counts for orders 10--25
\cite{DrakakisIorioRickard2011Order28,DrakakisIorioRickardWalsh2011Order29},
the same procedure gives a fitted quadratic coefficient of $0.017$.
The larger symmetric coefficient is qualitatively consistent with the
additional mirror correlation, but the value $0.036$ has no derivation.

Let
\[
E(n)=I(n)\exp\!\left[-\lambda_{\mathrm{model}}(n)
+r_{\mathrm{fit}}(n)\right]
\]
be the model's clumping-corrected estimate of the number of sporadic oriented
arrays.  Both the raw census fit and the corrected model place the crossing
$E(n)=1$ between orders 27 and 30.  No sporadic main-diagonal symmetric array
occurs at orders 24--42.
The difference between the
fitted crossing and the last observed sporadic order is within the
uncertainty of this heuristic calibration.  Beyond order 34 the
model estimate is below $10^{-1}$ and falls by roughly a factor 2
per order.  It is near $10^{-2}$ at order 38 and near $10^{-3}$ at orders
41 and 42.  The model therefore predicts no sporadic array at orders
37--42.  Section~\ref{sec:census} compares that prediction with the
exhaustive censuses.

\section{Exhaustive enumeration and verification}
\label{sec:method}

The census claims require complete enumeration.  The algorithm must generate
every involution exactly once and may prune a branch only when that branch
cannot contain a Costas array.  This section establishes these properties.
Appendix~\ref{app:repro} provides the validation protocols, completion
records, and checksums.

\textbf{Search algorithm.}  The solver constructs every involution exactly once,
orbit by orbit, and abandons a branch the moment it repeats a displacement
vector.  One recursion step proceeds as follows:

\begin{enumerate}
\item Let $a$ be the smallest row not yet assigned.
\item Branch on the partner of $a$: either $p_a=a$ (a fixed point), or
$p_a=b$ and $p_b=a$ for an unassigned $b>a$ (a transposition).  The
smallest-unassigned-row rule gives each involution a unique construction
path, so the solver enumerates no involution twice.  Every completed
assignment is an involution by construction, so no separate
diagonal-symmetry test is required.
\item Form every displacement vector between the new dot or dots and all
dots already placed, including the vector inside the new pair.  Reject
the choice if any of these vectors is already in use.  The constraints of
Section~\ref{sec:lemmas} are enforced with separate bitmasks for used centers
and used gaps.  A 128-bit mask for each row stride records the signed column
differences already used.  These masks allow the solver to check, apply, and
undo a proposed orbit with bitwise operations.
\item If the choice survives and unassigned rows remain, recurse.  When
no rows remain, the assignment is an involution with all displacement
vectors distinct, so it is a main-diagonal symmetric Costas array and is
written out.
\item Two sound early discards prune the tree.  First, near a leaf, the
solver checks a fixed number of the next unassigned rows and drops the
branch if any checked row has no feasible partner.  The number of rows
checked and the activation threshold affect the search cost but not the
soundness of the test.  Appendix~\ref{app:repro} records the production
settings.  Second, the solver drops a partial assignment once it proves
that the completed permutation would be the lexicographically larger member
of $\{p,\,RC(p)\}$.  The search retains one member of each reverse-complement
pair and regenerates the other at output.
\end{enumerate}

Keeping the differences signed matters.  Identifying $(k,\delta)$ with
$(k,-\delta)$ would reject valid arrays.  For $n>32$ the signed range
spans more than 64 values.  The production solver uses 128-bit masks and
is valid to $n=63$.

\textbf{Deterministic sharding.}  A shard records the sequence of partner
choices at a fixed orbit depth.  The CPU campaigns use depth three.  The GPU
campaigns at orders 37--42 use depth five.  Shard generation applies the same
acceptance predicates as the search.  The resulting subtrees are disjoint and
cover the canonical tree.  Appendix~\ref{app:repro} states and proves this
partition theorem.  Shards are executed independently.  A shard writes a
result only after successful completion.  An interrupted shard is rerun.
The source, binary, shard list, and result files are checksummed.

\textbf{Validation.}  The 64-bit and 128-bit builds agree on every
instrumented counter at orders 12--20, not only on solution counts.  Shard
unions reproduce the monolithic runs.  A separate CUDA kernel over depth-5
partitions reproduces the complete CPU census sets at orders 37--40.  Every
reported array is re-verified independently for the permutation, involution,
and Costas properties.  Each production configuration reproduces a reference
run before use.

\section{Exhaustive censuses at orders 37--42}
\label{sec:census}

\begin{theorem}
\label{thm:census}
The main-diagonal symmetric Costas arrays of orders 37 through 42 are
exactly the 38 arrays listed in Appendix~\ref{app:arrays}.  Their counts
at successive orders are 4, 0, 16, 2, 12, and 4.
\end{theorem}

\begin{proof}
Section~\ref{sec:method} gives the enumeration algorithm, and
Appendix~\ref{app:repro} proves that its branches cover every involution
without duplication and that each pruning rule is sound.  The CPU records
establish complete depth-3 shard coverage at orders 37--40.  A separate
CUDA kernel, run over depth-5 partitions, also completed those four orders
and returned the identical canonical representatives and
reverse-complement mates.  The CUDA records establish complete depth-5
coverage at orders 41 and 42.  Every reported permutation was checked
independently for the permutation, involution, and Costas properties.
\end{proof}

\begin{table}[ht]
\centering
\begin{tabular}{rrrl}
\toprule
order & arrays & $RC$ classes & classification \\
\midrule
37 & 4 & 2 & corner deletions of Lempel arrays over $\mathbb{F}_{41}$ \\
38 & 0 & 0 & empty \\
39 & 16 & 8 & Lempel arrays over $\mathbb{F}_{41}$ \\
40 & 2 & 1 & corner augmentations of Lempel arrays over $\mathbb{F}_{41}$ \\
41 & 12 & 6 & Lempel arrays over $\mathbb{F}_{43}$ \\
42 & 4 & 2 & two corner augmentations and the Rickard--Golomb pair \\
\bottomrule
\end{tabular}
\caption{Complete main-diagonal symmetric censuses at orders 37--42.
The 38 oriented arrays form 19 reverse-complement ($RC$) classes.}
\label{tab:census}
\end{table}

The sixteen order-39 arrays are the
sixteen Lempel arrays over $\mathbb{F}_{41}$.  For a primitive root $g$,
the corresponding Lempel array places a dot at $(i,j)$ when
$g^{i}+g^{j}=1$.  The four order-37 arrays are obtained by deleting a
two-row and two-column diagonal corner block from members of that family.
The two order-40 arrays are the diagonal-corner augmentations of one
reverse-complement class.  All 22 arrays match entries in the recorded
construction database \cite{BeardDataPort2017}; the census contains no
unclassified array.

No field construction reaches a main-diagonal
symmetric array of order 38, and the model of Section~\ref{sec:model} puts the
expected sporadic count there near $10^{-2}$.  The exhaustive census is
empty, consistent with both the construction record and the model.

The twelve order-41 Lempel arrays all have one fixed point.  A Lempel dot
at $(i,i)$ requires $2g^{i}\equiv 1 \pmod{43}$, or equivalently
$g^i\equiv22\pmod{43}$.  Because $g$ generates the nonzero elements of
$\mathbb{F}_{43}$, exactly one exponent $i$ modulo 42 satisfies this
equation.  Thus every Lempel array has exactly one diagonal dot,
equivalently one fixed point.  The twelve arrays form six
reverse-complement classes.  Exactly one class can be extended to order
42 by adding a diagonal corner dot; the two choices of corner produce the
two order-42 augmentations recorded in Appendix~\ref{app:arrays}.

Combining Table~\ref{tab:census} with the main-diagonal column of
\cite{RickardDrakakis2009Symmetry}, sporadic main-diagonal symmetric
Costas arrays occur at orders 11--23 and nowhere else through order 42.
The earlier census asked more broadly whether symmetric Costas arrays of
unknown origin exist above order 23.  The combined record answers that
question for main-diagonal symmetry through order 42.  The absence of a
sporadic array in the six new censuses is consistent with the model's
qualitative prediction.

\subsection{Order-42 census and the Rickard--Golomb family}
\label{sec:sporadic}

The two order-42 arrays that are not corner augmentations form the
reverse-complement pair
\begin{quote}\begin{footnotesize}\begin{verbatim}
19 1 41 25 30 37 40 32 36 15 10 20 39 24 38 9 29 18 17 0 11 26 34 27 13 3 21 23 35
   16 4 33 7 31 22 28 8 5 14 12 6 2
39 35 29 27 36 33 13 19 10 34 8 37 25 6 18 20 38 28 14 7 15 30 41 24 23 12 32 3 17
   2 21 31 26 5 9 1 4 11 16 0 40 22
\end{verbatim}\end{footnotesize}\end{quote}
Rows and columns are numbered from 0 to 41.  Each permutation is split
across two display lines; concatenating the two lines gives its 42
entries.  Each array has 20
transpositions and two fixed points --- at indices 1 and 10 in the first
array, at their reflections 31 and 40 in the second --- and no dot in
any corner cell.

To determine whether this pair has a known algebraic origin, we test the
symmetric Rickard--Golomb family $RG_{1}$
\cite{Rickard2004Periodicity,Drakakis2011OpenProblems}.  For a prime power
$q$, choose a primitive element $g$ of $\mathbb{F}_{q}$ and a nonzero
constant $c$.  Place a dot at
$(i,j)\in\mathbb{Z}_{q-1}\times\mathbb{Z}_{q-1}$ when
$c\,(g^{i}+g^{j})=1$.  The relation leaves exactly one row and one column
empty, both indexed by the exponent $i_{0}$ for which
$g^{i_{0}}=c^{-1}$.  Adding the dot $(i_{0},i_{0})$ produces a symmetric
permutation candidate of order $q-1$; the candidate need not have the
Costas property.  The choice $c=1$ gives the corner-augmented Lempel
candidates.  Proposition~\ref{prop:family} tests the complete family at
$q=43$.

\begin{proposition}
\label{prop:family}
Of the 504 arrays of this family at $q=43$, exactly four have the Costas
property, and they are the four census arrays: the corner augmentations
at $(g,c)=(19,1)$ and $(34,34)$, and the pair above at $(30,33)$ and
$(33,14)$.
\end{proposition}

\begin{proof}
The verification enumerates all twelve primitive roots and all 42 nonzero
values of $c$.  For each of the resulting 504 parameter pairs, it
constructs the permutation and checks every row of its difference
triangle for a repeated value.  Exactly the four stated parameter pairs
pass.  The source and complete output are included with the computational
artifacts.
\end{proof}

The pair is a different $RG_{1}$ equivalence class from the order-42 array
reported by Rickard in 2004 and recorded in the survey
\cite{Rickard2004Periodicity,Drakakis2011OpenProblems}.  Unlike the other
two order-42 arrays,
neither member is a direct corner augmentation: neither has a corner dot,
and a finite check finds no one-dot corner extension with the Costas
property.  Thus, within the main-diagonal class, the census extends through
order 42 the negative record underlying the closing question of
\cite{RickardDrakakis2009Symmetry}.

The pair was found by the GPU enumeration and verified twice more:
the CPU solver of Section~\ref{sec:method}, restricted to the subtree
with orbit prefix $(19,1,41,25,30)$, reports exactly these two arrays,
and both satisfy the permutation, involution and distinct-vector
conditions by direct check.

\section{Discussion}
\label{sec:discussion}

Exhaustive results now determine the main-diagonal symmetric class
through order 42.  The census contains no sporadic array at any of the
nineteen consecutive orders from 24 through 42.  The collision model is
consistent with this record.  Its $n^2/18$ and $n^{3/2}/360$ terms are
derived.  Its linear-scale remainder and clumping correction are calibrated.

At orders 24 through 42, existence in the main-diagonal symmetric class
coincides with the reach of the known constructions: the Lempel family,
its corner edits, and the Rickard--Golomb family.  The full Costas census
shows a similar numerical decline, from 21{,}104 arrays at order 16 to 56
at order 26 \cite{DrakakisIorioRickardWalsh2011Order29}, but the symmetric
model does not by itself establish the corresponding law for general
permutations.

The $O(n)$ term in Theorem~\ref{thm:first-moment} remains open.
The covariance contribution to the self-conjugate linear coefficient is
presently heuristic, and the fitted clumping coefficient $0.036$ has no
derivation.  Deriving these two coefficients would sharpen the modeled
crossing near orders 27--30 but would not affect the exact censuses.

The enumeration does not use the linear-forest representation of
Proposition~\ref{prop:forest}.  Appendix~\ref{app:negative} tests several
search strategies and topology statistics derived from this representation.
None improved the solver at the tested orders or separated the census arrays
from matched random involutions.  Stronger graph-theoretic constraints may
nevertheless lead to better search methods.

\appendix

\section{Evaluation of alternative search strategies}
\label{app:negative}

We evaluated whether the linear-forest representation of
Proposition~\ref{prop:forest} could reduce the enumeration cost.  We
implemented three natural search strategies and compared each with the
solver of Section~\ref{sec:method} under matched conditions.  None
reduced the search cost at the tested orders.  The companion repository
contains the source code and compact measurement tables.

\begin{enumerate}
\item \textbf{Partition by fixed-point count.}  The admissible fixed-point
classes can be searched independently, so the split appears inexpensive.
Measured at order 20 it needs 23 times the nodes of the single tree: the
classes cannot share partial results, and the same transposition-heavy
prefixes are re-explored in every class.
\item \textbf{Search by edge length.}  Each gap occurs at most once
(Lemma~\ref{lem:antidiag}), so one can branch on ``does a transposition
of gap $g$ exist, and where?'', longest first.  At order 20 this visits
282 times as many states as the row-by-row search; the ratio also increases
across the tested orders 16, 18, and 20.  Long gaps place isolated dots whose mutual
constraints are weak, so early decisions run nearly unconstrained, while
row-by-row placement packs dots densely and every new orbit is tested
against everything already placed.
\item \textbf{Feasibility checks on the remaining rows.}  We tested whether
enough unused centers and gaps remained for the unassigned rows.  The test
rejected only three of $5.7\times 10^{7}$ states in an order-38 subtree.
Near the leaves, the unused centers and gaps greatly outnumber the unassigned
rows.  Exact completion can be formulated as a rainbow-matching problem, but
this counting test provides almost no pruning.
\end{enumerate}

To test whether Costas involutions have unusual forest shapes, we compared
the 486 census arrays of orders 7--26 and the arrays of
Section~\ref{sec:census} with random involutions.  Each control group
contained $2{,}000$ involutions with the same order and fixed-point count
as the corresponding census group.  We measured the number of chained
orbit pairs, the longest path, and the number of components.  No matched
group differed from its control by more than $2.4$ estimated standard
errors, so none of the three statistics separates the two populations at
an individual order and fixed-point count.

Pooling the chained-orbit counts gives $Z=+3.6$, but this apparent
difference is confined to orders at most 18 ($Z=+4.6$) and is absent at
later orders ($Z=-1.7$).  The reverse-complement mates and the pooled
groups are not independent, so these values do not establish statistical
significance.  The longest path in any census array has five edges.

This study yielded no sound pruning rule based only on topology.  The
available spectrum-level exclusion is the edge-count bound implied by
Corollary~\ref{prop:ruler}: the topology determines
$e=(n+f)/2$, and the Golomb-ruler condition bounds $f$.  The remaining
constraints used by the solver depend on the vertex labels and are not
captured by the tested unlabeled forest statistics.

\section{Reproducibility}
\label{app:repro}

\subsection{Correctness of the enumeration}

The solver maintains five invariants.  Together they prove
Theorem~\ref{thm:census}, provided that the implementation executes the
stated rules correctly.  Sections~\ref{app:validation}
and~\ref{app:artifacts} give the corresponding execution checks.

\begin{enumerate}
\item \textbf{Partial involution.} Every assigned index belongs to exactly
one fixed point or transposition; both orbit indices must be unassigned
when the orbit is placed.
\item \textbf{Exact difference state.} For every stride, the difference
mask contains exactly the signed column differences among currently
assigned dots; apply and rollback operate on collision-free grouped masks.
\item \textbf{No repeated vector.} Every new vector is tested against the
permanent masks and against the same-orbit pending mask before
commitment.
\item \textbf{Sound pruning.} The center mask, the gap mask, the
empty-domain forward check, and reverse-complement canonicalization remove
only branches that repeat a vector, have a row with no candidate, or are
the lexicographically larger member of a rotation pair.  The
reverse-complement rule, precisely: scan $i=0,1,\ldots$; if $p_{i}$ or
$p_{n-1-i}$ is unassigned, the comparison is undecided and the branch is
kept; otherwise the first index with $p_{i}\neq n-1-p_{n-1-i}$ decides,
and the branch is pruned exactly when $p_{i}> n-1-p_{n-1-i}$ there.  A
subtree decided smaller is not rechecked; a full assignment with all
comparisons equal is its own reverse complement and is reported once.
\item \textbf{Exhaustive partition.} Least-unassigned-row orbit branching
gives every involution exactly one construction path.
\end{enumerate}

\begin{proposition}[shard partition]
\label{prop:shards}
For any fixed orbit depth, the shard subtrees are pairwise disjoint and
their union is the complete canonical search tree.
\end{proposition}

\begin{proof}
Shard generation runs the identical acceptance predicates as search and
emits the partner sequence at the chosen orbit depth.  By invariant~5 every
surviving leaf has exactly one prefix of that length; distinct prefixes diverge
at an orbit choice and cannot reach the same involution; prefix replay
revalidates every orbit and reconstructs the identical frontier state.
\end{proof}

\begin{proposition}[soundness and completeness]
Every reported permutation is a main-diagonal symmetric Costas array, and
every such array of the requested order is reported exactly once.
\end{proposition}

\begin{proof}
By invariants 1--3 every accepted leaf is an involution with distinct
displacement vectors.  By invariants 4--5 no valid canonical branch is
pruned and no branch is visited twice; the reverse complement of each
canonical leaf is regenerated at output when distinct.
\end{proof}

\subsection{Validation}
\label{app:validation}

\begin{itemize}
\item \textbf{Counter equality.} The 64-bit and 128-bit implementations
agree on every instrumented counter, not only on solution counts.  Each
entry below is the common value produced by both implementations.
``Arrays'' counts oriented outputs.  A recursive state is a partial
involution entered by the search.  A candidate orbit is a proposed fixed
point or transposition, and a locally valid orbit is one that passes the
immediate difference, center, and gap checks.  The final two columns count
locally valid orbits subsequently removed by the lookahead and
reverse-complement tests.

\begin{center}
\begin{tabular}{rrrrrrr}
\toprule
 & & recursive & candidate & locally valid & lookahead & $RC$ \\
$n$ & arrays & states & orbits & orbits & prunes & prunes \\
\midrule
12 & 34 & 2{,}379 & 5{,}241 & 5{,}119 & 2{,}472 & 269 \\
16 & 40 & 57{,}549 & 132{,}318 & 128{,}139 & 65{,}419 & 5{,}172 \\
20 & 8 & 1{,}590{,}471 & 3{,}691{,}067 & 3{,}558{,}475 & 1{,}849{,}278 & 118{,}727 \\
\bottomrule
\end{tabular}
\end{center}

\item \textbf{Shard union.} At order 20, all 4{,}052 depth-3 shards
completed and their union equals the monolithic solver's array set.
\item \textbf{Sanitizers.} An order-12 run completed under AddressSanitizer
and UndefinedBehaviorSanitizer without a report.
\item \textbf{Reference counters.} Before each campaign, every production
configuration reproduced the order-16 reference counters (40 arrays,
57{,}549 states, 132{,}318 candidate tests).
\item \textbf{Independent leaf validation.} Every reported array was
reconstructed from the raw result files and re-checked by separate code
for permutation completeness, the involution identity, and signed
Costas-difference uniqueness.
\end{itemize}

\subsection{Production runs}

The CPU censuses used two lookahead configurations.  At orders 37 and
38, the solver checked the next two unassigned rows when at most eight
rows remained.  At orders 39 and 40, it checked the next four rows when
at most nine remained.  The archived manifests record the source,
executable, compiler configuration, and execution environment for every
run.

\begin{center}
\begin{tabular}{lrrrr}
\toprule
 & $n=37$ & $n=38$ & $n=39$ & $n=40$ \\
\midrule
depth-3 shards & 35{,}552 & 38{,}891 & 42{,}431 & 46{,}182 \\
failed shards & 0 & 0 & 0 & 0 \\
search states & $8.77\times 10^{12}$ & $2.06\times 10^{13}$ & $2.72\times 10^{13}$ & $6.39\times 10^{13}$ \\
candidate tests & $1.59\times 10^{13}$ & $3.73\times 10^{13}$ & $5.22\times 10^{13}$ & $1.23\times 10^{14}$ \\
reported arrays & 4 & 0 & 16 & 2 \\
\bottomrule
\end{tabular}
\end{center}

\subsection{GPU implementation and validation}
\label{app:gpu}

A CUDA implementation re-enumerated orders 37--40 and produced the
censuses at orders 41 and 42.  It uses the same orbit tree, difference
masks, center and gap constraints,
reverse-complement canonicalization with mate regeneration, and four-row
lookahead activated when at most nine rows remain, with one warp owning
one depth-5 shard subtree and all warp state in shared memory.  The runs
at orders 37--40 are complete
re-enumerations of the CPU censuses; the runs at orders 41 and 42 establish
the new censuses.  Because the branching and pruning rules are identical,
the soundness and completeness argument in Appendix~\ref{app:repro}
applies to the GPU search.  The checks below address faithful execution of
the GPU implementation.

\begin{enumerate}
\item \textbf{Census gates.}  The GPU reproduces the exact censuses at
orders 12, 14 and 16 (34, 46, 40 arrays) in three configurations:
canonicalization off, on, and on with lookahead.
\item \textbf{Order 30.}  The full GPU census (8 arrays) matches a complete
census by the CPU solver of Section~\ref{sec:method}.
\item \textbf{Order 34.}  The full GPU census gives two oriented arrays,
corresponding to the single diagonal equivalence class in the published
record \cite{RickardDrakakis2009Symmetry}.  It was run twice, at shard
depths 4 and 5, with identical output.
\item \textbf{Complete CPU comparison.}  The depth-5 GPU runs at orders
37--40 completed every shard and reproduced the exact canonical leaf sets,
and hence the oriented array sets, of the depth-3 CPU censuses.
\item \textbf{Completion records.}  Each of the six runs has a per-shard
completion bitmap with every bit set.  The shard counts are given in
Table~\ref{tab:censusgpu}.
\item \textbf{Cross-verification.}  The subtree containing the order-42 Rickard--Golomb pair was
re-run by the CPU solver, which reports the identical pair
(Section~\ref{sec:sporadic}).
\end{enumerate}

\begin{table}[ht]
\centering
\begin{tabular}{rrrrr}
\toprule
order & depth-5 shards & search nodes & wall time (h) & arrays \\
\midrule
37 & 18{,}362{,}447 & $1.82\times 10^{13}$ & 1.39 & 4 \\
38 & 21{,}812{,}062 & $4.29\times 10^{13}$ & 3.31 & 0 \\
39 & 25{,}757{,}502 & $1.01\times 10^{14}$ & 7.75 & 16 \\
40 & 30{,}263{,}534 & $2.36\times 10^{14}$ & 18.1 & 2 \\
\midrule
41 & 35{,}374{,}540 & $5.52\times 10^{14}$ & 43.5 & 12 \\
42 & 41{,}163{,}548 & $1.29\times 10^{15}$ & 102.4 & 4 \\
\bottomrule
\end{tabular}
\caption{Complete GPU execution records.  Orders 37--40 are full
re-enumerations of the CPU censuses; orders 41 and 42 are census runs.
Every shard completed, as certified by the per-shard completion bitmaps.}
\label{tab:censusgpu}
\end{table}

The GPU runs used one NVIDIA GeForce RTX 5090 with CUDA 12.8.

\subsection{Artifacts}
\label{app:artifacts}

Every campaign stored SHA-256 checksums for source, executable, runner,
shard lists, and atomic results; every collected manifest was rechecked.
The companion repository distributes the exact CPU and GPU source
revisions and runners, the six censuses, compressed GPU run logs and
completion bitmaps, historical checksum records, compact validation and
measurement tables, the programs used for the model, and the complete
fixed-seed exact and Monte Carlo series with standard errors, fit inputs,
and fitted coefficients.  Generated shard
frontiers, atomic shard outputs, deployment archives, and compiled binaries
are not distributed; the frontiers can be regenerated from the included
CPU source.\footnote{\url{https://github.com/bogdan27182/costas-md-paper}}

\subsection{Computations for the model}
\label{app:model-computation}

The exponent $\lambda_{\mathrm{sym}}$ was computed exactly for
$n\leq 18$ by enumeration over all $I(n)$ involutions.  The Monte Carlo
computations used $K=2\times 10^{5}$ samples per order at $n=19$--$48$ and
$K=2\times 10^{6}$ samples per order at $n=40$--$60$.  An additional
$K=2\times10^6$ run at $n=19$--$23$ checked the Monte Carlo sensitivity of
the clumping fit.  All runs used a fixed seed.  The sampler draws uniform
involutions using the
exact ratio $I(m-1)/I(m)$, and the resulting standard errors are below
$0.08$.  The general-permutation clumping comparison uses the exact first
moment and the published census counts at orders 10--25.  The companion
repository contains the moment and fitting programs, census inputs, the
complete numerical series, standard errors, fit inputs, and fitted
coefficients.

For the class $(k,-k)$, the occupancy $m_k$ counts stride-$k$ index pairs
whose dots share a center sum.  Each stride offers $n-k$ index pairs.
Among the $\sim n^2$ generic assignments of the two values, $n-k$ have
the required difference.  The internal transposition is a separate
configuration of probability of order $1/n$, and therefore has weight $n$
on this $n^2$ scale.  Thus
\[
\mathbb{E}[m_k]
=(n-k)\frac{n+(n-k)}{n^{2}}+\text{lower order}
=\frac{(n-k)(2n-k)}{n^{2}}+\text{lower order}.
\]
The diagonal classes $(k,k)$ obey the same count, with two fixed points
replacing the internal transposition.  Treating the occupancies as independent
and summing over both families and all strides, with $j=n-k$, gives
\begin{align*}
\mathbb{E}[W_{\mathrm{self}}]
&\approx \sum_{k=1}^{n-1}\frac{\bigl[(n-k)(2n-k)\bigr]^{2}}{n^{4}}\\
&=\sum_{j=1}^{n-1}\frac{j^{2}n^{2}+2nj^{3}+j^{4}}{n^{4}}
=\frac{31}{30}n+O(1).
\end{align*}
The Poisson approximation does not account for the whole linear term.
The occupancy indicators are correlated by the same transpose closure
that drives Lemma~\ref{lem:antidiag}.  If rows $(i,i+k)$ carry values
$(a,a-k)$, then rows $(a-k,a)$ carry values $(i+k,i)$, which is another
occurrence of the same class.  The resulting pairing of generic
occurrences motivates a covariance correction of
$\mathbb{E}[m_{\mathrm{gen}}]/2$ per class.  Summed over both families,
the heuristic correction is
$\sum_k(n-k)^2/n^2=n/3+O(1)$.  The modeled self-conjugate linear
coefficient is therefore
\[
c_{\mathrm{self}}=\frac{31}{30}+\frac13=\frac{41}{30}\approx1.367.
\]

Both finite-range fits use the exact values at $n=12,14,16,18$ and the
$K=2\times10^6$ Monte Carlo estimates at every order from 40 through 60.
The exact inputs are the four largest available even orders; this keeps
parity fixed and excludes smaller orders, where omitted terms are relatively
larger.  The fits give each order equal weight because they describe an
incomplete finite-range expansion.  Weighting only by sampling error would
make the exact small-order values dominate.  The separate fit
\[
\mathbb{E}[W_{\mathrm{self}}]=a_1n+a_2\sqrt n+a_3
\]
gives $a_1=1.362968$, $a_2=-2.281745$, and $a_3=0.641259$.  Exact
enumeration at $n=12$ and 14 gives covariance contributions of 2.37 and
2.97, respectively, consistent with the modeled slope of approximately
$1/3$ per order.

For the three-term clumping fit, using $K=2\times10^5$ at orders 19--23
gives a quadratic coefficient of $0.0348$.  Increasing those orders to
$K=2\times10^6$ gives $0.0358$, reported as $0.036$.  One-parameter fits
of the same log-ratio to a multiple of $n^2$ give $0.037$ from the total
counts and $0.033$ from the sporadic-only counts after conversion to the
oriented convention.

\section{Enumerated arrays}
\label{app:arrays}

All permutations are zero-based.

\vspace{2.5em}

\noindent Order 37 (census, 4 arrays):\par\nopagebreak
\begingroup\small
\noindent\begin{tabular}{@{}r@{\hskip 8pt}>{\ttfamily\raggedright\arraybackslash}p{0.9\textwidth}@{}}
\toprule
1 & 13 8 31 28 21 15 32 19 1 14 26 17 20 0 9 5 35 11 36 7 12 4 25 23 33 22 10 29 3 27 34 2 6 24 30 16 18\\
\addlinespace
2 & 18 20 6 12 30 34 2 9 33 7 26 14 3 13 11 32 24 29 0 25 1 31 27 36 16 19 10 22 35 17 4 21 15 8 5 28 23\\
\addlinespace
3 & 18 20 34 30 27 10 35 17 32 21 5 36 16 19 28 22 12 7 0 13 1 9 15 26 33 25 23 4 14 31 3 29 8 24 2 6 11\\
\addlinespace
4 & 25 30 34 12 28 7 33 5 22 32 13 11 3 10 21 27 35 23 36 29 24 14 8 17 20 0 31 15 4 19 1 26 9 6 2 16 18\\
\bottomrule
\end{tabular}
\endgroup

\vspace{2.5em}

\noindent Order 38 (census): none.

\vspace{2.5em}

\noindent Order 39 (census, 16 arrays):\par\nopagebreak
\begingroup\small
\noindent\begin{tabular}{@{}r@{\hskip 8pt}>{\ttfamily\raggedright\arraybackslash}p{0.9\textwidth}@{}}
\toprule
1 & 24 16 10 38 22 15 13 8 7 23 2 28 25 6 35 5 1 19 30 17 31 21 4 9 0 12 32 36 11 33 18 20 26 29 37 14 27 34 3\\
\addlinespace
2 & 35 4 11 24 1 9 12 18 20 5 27 2 6 26 38 29 34 17 7 21 8 19 37 33 3 32 13 10 36 15 31 30 25 23 16 0 28 22 14\\
\addlinespace
3 & 1 0 27 32 36 14 30 9 35 7 24 34 15 13 5 12 23 29 37 25 38 31 26 16 10 19 22 2 33 17 6 21 3 28 11 8 4 18 20\\
\addlinespace
4 & 18 20 34 30 27 10 35 17 32 21 5 36 16 19 28 22 12 7 0 13 1 9 15 26 33 25 23 4 14 31 3 29 8 24 2 6 11 38 37\\
\addlinespace
5 & 28 24 37 4 3 9 13 38 16 5 31 22 25 6 36 29 8 17 34 21 35 19 11 33 1 12 32 30 0 15 27 10 26 23 18 20 14 2 7\\
\addlinespace
6 & 31 36 24 18 20 15 12 28 11 23 38 8 6 26 37 5 27 19 3 17 4 21 30 9 2 32 13 16 7 33 22 0 25 29 35 34 1 14 10\\
\addlinespace
7 & 1 0 15 10 33 30 23 17 34 21 3 16 28 19 22 2 11 7 37 13 38 9 14 6 27 25 35 24 12 31 5 29 36 4 8 26 32 18 20\\
\addlinespace
8 & 18 20 6 12 30 34 2 9 33 7 26 14 3 13 11 32 24 29 0 25 1 31 27 36 16 19 10 22 35 17 4 21 15 8 5 28 23 38 37\\
\addlinespace
9 & 28 1 5 8 17 2 21 38 3 13 31 18 20 9 30 25 33 4 11 37 12 6 36 29 35 15 27 26 0 23 14 10 34 16 32 24 22 19 7\\
\addlinespace
10 & 31 19 16 14 6 22 4 28 24 15 38 12 11 23 3 9 2 32 26 1 27 34 5 13 8 29 18 20 7 25 35 0 17 36 21 30 33 37 10\\
\addlinespace
11 & 27 19 5 34 28 2 38 8 7 15 35 32 17 23 21 9 33 12 30 1 31 14 36 13 26 29 24 0 4 25 18 20 11 16 3 10 22 37 6\\
\addlinespace
12 & 32 1 16 28 35 22 27 18 20 13 34 38 14 9 12 25 2 24 7 37 8 26 5 29 17 15 21 6 3 23 31 30 0 36 10 4 33 19 11\\
\addlinespace
13 & 3 29 6 0 12 5 2 37 17 31 21 26 4 14 13 32 35 8 24 33 25 10 38 36 18 20 11 34 30 1 28 9 15 19 27 16 23 7 22\\
\addlinespace
14 & 16 31 15 22 11 19 23 29 10 37 8 4 27 18 20 2 0 28 13 5 14 30 3 6 25 24 34 12 17 7 21 1 36 33 26 38 32 9 35\\
\addlinespace
15 & 25 29 18 20 38 5 11 37 10 31 8 6 15 34 27 12 23 28 2 33 3 30 26 16 32 0 22 14 17 1 21 9 24 19 13 36 35 7 4\\
\addlinespace
16 & 34 31 3 2 25 19 14 29 17 37 21 24 16 38 6 22 12 8 35 5 36 10 15 26 11 4 23 32 30 7 28 1 27 33 0 18 20 9 13\\
\bottomrule
\end{tabular}
\endgroup

\vspace{2.5em}

\noindent Order 40 (census, 2 arrays):\par\nopagebreak
\begingroup\small
\noindent\begin{tabular}{@{}r@{\hskip 8pt}>{\ttfamily\raggedright\arraybackslash}p{0.9\textwidth}@{}}
\toprule
1 & 0 33 2 17 29 36 23 28 19 21 14 35 39 15 10 13 26 3 25 8 38 9 27 6 30 18 16 22 7 4 24 32 31 1 37 11 5 34 20 12\\
\addlinespace
2 & 27 19 5 34 28 2 38 8 7 15 35 32 17 23 21 9 33 12 30 1 31 14 36 13 26 29 24 0 4 25 18 20 11 16 3 10 22 37 6 39\\
\bottomrule
\end{tabular}
\endgroup

\vspace{2.5em}

\noindent Order 41 (census, 12 arrays):\par\nopagebreak
\begingroup\small
\noindent\begin{tabular}{@{}r@{\hskip 8pt}>{\ttfamily\raggedright\arraybackslash}p{0.9\textwidth}@{}}
\toprule
1 & 2 30 0 29 37 16 34 39 8 24 33 36 17 14 13 23 5 12 38 26 32 27 40 15 9 28 19 21 25 3 1 35 20 10 6 31 11 4 18 7 22\\
\addlinespace
2 & 18 33 22 36 29 9 34 30 20 5 39 37 15 19 21 12 31 25 0 13 8 14 2 28 35 17 27 26 23 4 7 16 32 1 6 24 3 11 40 10 38\\
\addlinespace
3 & 25 13 20 40 35 31 34 17 27 36 23 12 11 1 16 37 14 7 24 28 2 29 26 10 18 0 22 8 19 21 33 5 39 30 6 4 9 15 38 32 3\\
\addlinespace
4 & 37 8 2 25 31 36 34 10 1 35 7 19 21 32 18 40 22 30 14 11 38 12 16 33 26 3 24 39 29 28 17 4 13 23 6 9 5 0 20 27 15\\
\addlinespace
5 & 23 40 36 30 28 37 34 14 20 11 35 9 38 26 7 19 21 39 29 15 8 16 31 0 25 24 13 33 4 18 3 22 32 27 6 10 2 5 12 17 1\\
\addlinespace
6 & 39 23 28 35 38 30 34 13 8 18 37 22 36 7 27 16 15 40 9 24 32 25 11 1 19 21 33 14 2 31 5 29 20 26 6 3 12 10 4 0 17\\
\addlinespace
7 & 30 35 9 19 21 28 34 11 40 2 32 7 39 31 14 17 29 15 23 3 26 4 25 18 33 22 20 38 5 16 0 13 10 24 6 1 37 36 27 12 8\\
\addlinespace
8 & 32 28 13 4 3 39 34 16 30 27 40 24 35 2 20 18 7 22 15 36 14 37 17 25 11 23 26 9 1 33 8 38 0 29 6 12 19 21 31 5 10\\
\addlinespace
9 & 29 12 2 18 14 22 34 40 15 28 26 33 1 25 4 8 17 16 3 23 38 24 5 19 21 13 10 32 9 0 36 39 27 11 6 37 30 35 20 31 7\\
\addlinespace
10 & 33 9 20 5 10 3 34 29 13 1 4 40 31 8 30 27 19 21 35 16 2 17 37 24 23 32 36 15 39 7 14 12 25 0 6 18 26 22 38 28 11\\
\addlinespace
11 & 5 17 37 27 23 0 34 18 12 22 19 21 8 38 14 35 28 1 7 10 26 11 9 4 32 40 20 3 16 30 29 33 24 31 6 15 39 2 13 36 25\\
\addlinespace
12 & 15 4 27 38 1 25 34 9 16 7 11 10 24 37 20 0 8 36 31 29 14 30 33 39 12 5 26 2 32 19 21 18 28 22 6 40 17 13 3 23 35\\
\bottomrule
\end{tabular}
\endgroup

\vspace{2.5em}

\noindent Order 42 (census, 4 arrays; 1 and 2 are the corner augmentations, 3 and 4 the Rickard--Golomb pair of Section~\ref{sec:sporadic}):\par\nopagebreak
\begingroup\small
\noindent\begin{tabular}{@{}r@{\hskip 8pt}>{\ttfamily\raggedright\arraybackslash}p{0.9\textwidth}@{}}
\toprule
1 & 0 38 9 3 26 32 37 35 11 2 36 8 20 22 33 19 41 23 31 15 12 39 13 17 34 27 4 25 40 30 29 18 5 14 24 7 10 6 1 21 28 16\\
\addlinespace
2 & 25 13 20 40 35 31 34 17 27 36 23 12 11 1 16 37 14 7 24 28 2 29 26 10 18 0 22 8 19 21 33 5 39 30 6 4 9 15 38 32 3 41\\
\addlinespace
3 & 19 1 41 25 30 37 40 32 36 15 10 20 39 24 38 9 29 18 17 0 11 26 34 27 13 3 21 23 35 16 4 33 7 31 22 28 8 5 14 12 6 2\\
\addlinespace
4 & 39 35 29 27 36 33 13 19 10 34 8 37 25 6 18 20 38 28 14 7 15 30 41 24 23 12 32 3 17 2 21 31 26 5 9 1 4 11 16 0 40 22\\
\bottomrule
\end{tabular}
\endgroup

\bibliographystyle{plainnat}
\bibliography{references}

\end{document}